\documentclass{amsart}

\usepackage{amssymb}
\usepackage[all]{xy}
\usepackage{palatino}

\def\eu{\mathfrak}
\def\ma{\mathbb}
\def\mc{\mathcal}

\def\S#1{S_{\infty}(#1)}
\def\e#1{e_{\infty}(#1)}
\def\f#1{f_{\infty}(#1)}
\def\p{{\mathcal P}_{\infty}}
\def\g#1{#1_{{\eu {ge}}}}

\def\cicl#1#2{k(\Lambda_{#1^{#2}})}
\def\F{{\ma F}_q}
\def\fin{\hfill\qed\bigskip}

\def\Witt#1{\stackrel{_{\bullet}}{#1}}

\newcommand{\Gal}{\operatorname{Gal}}
\newcommand{\lcm}{\operatorname{lcm}}

\newcounter{bean}

\def\l{
\begin{list}
{\rm{(\alph{bean}).-}}{\usecounter{bean}
\setlength{\labelwidth}{0.8in}
\setlength{\labelsep}{0.3cm}
\setlength{\leftmargin}{1cm}}}

\numberwithin{equation}{section}
\newtheorem{theorem}{Theorem}[section]
\newtheorem{proposition}[theorem]{Proposition}
\newtheorem{remark}[theorem]{Remark}
\newtheorem{remarks}[theorem]{Remarks}
\newtheorem{corollary}[theorem]{Corollary}
\newtheorem{example}[theorem]{Example}

\title[Genus fields of finite abelian extensions]
{Genus fields of finite abelian extensions}
 
\author[J. Barreto]{Jonny Fernando Barreto--Casta\~neda}
\address{Departamento de Control Autom\'atico\\
Centro de Investigaci\'on y de Estudios Avanzados del I.P.N.}
\email{jbarreto@ctrl.cinvestav.mx}
 
\author[C. Montelongo]{Carlos Montelongo--V\'azquez}
\address{Departamento de Control Autom\'atico\\
Centro de Investigaci\'on y de Estudios Avanzados del I.P.N.}
\email{cmontelongo@ctrl.cinvestav.mx}

\author[C. Reyes]{Carlos Daniel Reyes--Morales}
\address{Departamento de Control Autom\'atico\\
Centro de Investigaci\'on y de Estudios Avanzados del I.P.N.}
\email{mcenigm@gmail.com}

\author[M. Rzedowski]{Martha Rzedowski--Calder\'on}
\address{Departamento de Control Autom\'atico\\
Centro de Investigaci\'on y de Estudios Avanzados del I.P.N.}
\email{mrzedowski@ctrl.cinvestav.mx}

\author[G. Villa]
{Gabriel Villa--Salvador}
\address{
Departamento de Control Autom\'atico\\
Centro de Investigaci\'on y de Estudios Avanzados del I.P.N.}
\email{gvillasalvador@gmail.com, gvilla@ctrl.cinvestav.mx}



\date{October 21st., 2017}

\begin{document}

\begin{abstract}

In this paper we find the genus field of finite abelian
extensions of the global rational function field. 
We introduce the term conductor of constants for these
extensions and determine it in terms of other invariants.
We study the particular
case of finite abelian $p$--extensions and
give an explicit description of their genus field.

\end{abstract}

\maketitle

\section{Introduction}\label{S1}

It was C. F. Gauss \cite{Gau1801} the first one to consider what
now is known as the {\em genus field}. The work of Gauss
was in the
context of binary quadratic forms. Later on this
concept was translated into the context 
of quadratic number fields. In this way,
originally, the definition of
genus field was given for a quadratic extension of ${\ma Q}$. 
We have that for a quadratic number field $K$, the Galois group of
$K_{{\eu {ge}}}/K$, $\g K$ denoting the genus field of $K$,
is isomorphic to the maximal subgroup of 
exponent $2$ of the
ideal class group of $K$. It was proved by Gauss
that if $s$ is the number of different positive finite
rational primes dividing the discriminant $\delta_K$ of a quadratic
number field $K$, then the $2$--rank of the class group of $K$
is $2^{s-2}$ if $\delta_K>0$ and there exists a prime $p\equiv
3\bmod 4$ dividing $\delta_K$ and $2^{s-1}$ otherwise.

Genus theory using class field theory was introduced
by H. Hasse \cite{Has51} for the special case of
quadratic number fields. Hasse translated Gauss' genus theory
using characters. H. W. Leopoldt \cite{Leo53}
generalized the results of Hasse determining
the genus field $\g K$ of an absolute abelian
number field $K$. Leopoldt used {\em Dirichlet characters}
to develop genus theory of absolute 
abelian extensions and related the theory
of Dirichlet characters to the arithmetic of $K$. 

The concept of genus fields for an arbitrary
finite extension of the field of rational numbers was
introduced by A. Fr\"ohlich \cite{Fro59-1, Fro59-2, Fro83}. 
Fr\"ohlich defined the genus field $\g K$ of an arbitrary
finite number field $K/{\ma Q}$ as $\g K:= Kk^{\ast}$ where
$k^{\ast}$ is the maximal abelian number field such that $Kk^{\ast}/
K$ is unramified. We have that $k^{\ast}$ is the maximal abelian
number field contained in $\g K$. The degree $[\g K:K]$ is called
the {\em genus number} of $K$ and the Galois group $\Gal(
\g K/K)$ is called the {\em genus group} of $K$.

We have that if $K_H$ denotes the {\em Hilbert class field} of $K$, 
then $K\subseteq \g K\subseteq K_H$ and $\Gal(K_H/K)$ is
isomorphic to the class group $Cl_K$ of $K$. The
genus field $\g K$
corresponds to a subgroup $G_K$ of $Cl_K$, that is,
$\Gal(\g K/K)\cong Cl_K/G_K$. The subgroup $G_K$ is called
the {\em principal genus} of $K$ and $|Cl_K/G_K|$ is equal to
the genus number of $K$.

X. Zhang \cite{Xia85} gave a simple expression
of $K_{{\eu {ge}}}$ for any abelian extension $K$ 
of ${\ma Q}$ using Hilbert ramification theory. 
M. Ishida \cite{Ish76} described the {\em narrow genus field}
$\g K$ of any finite extension of ${\ma Q}$. That is, 
Ishida allowed ramification at the infinite primes. 
Given a number field $K$, Ishida found two
abelian number fields $k_1^{\ast}$ and $k_2^{\ast}$ 
such that $k^{\ast}=k_1^{\ast}k_2^{\ast}$ and
$k_1^{\ast}\cap k_2^{\ast}={\ma Q}$. The field $k_1^{\ast}$
is related to the finite primes $p$ such that
at least one prime in $K$ above $p$ is tamely ramified.

We are interested in genus theory for global function fields.
There is no direct proper notion of
Hilbert class field because, since all the constant field extensions are abelian
and unramified, the maximal constant extension is infinite
abelian and unramified. On the other
extreme, if the class number of a congruence function
field $K$ is $h_K$ then there are exactly $h:=h_K$ abelian
extensions $K_1,\ldots, K_h$ of $K$ such that $K_i/K$ are maximal
unramified with exact field of constants of each $K_i$ the same
as the one of $K$, ${\ma F}_q$, the finite field of $q$ elements
and $\Gal(K_i/K)\cong Cl_{K,0}$ the group of classes of
divisors of degree zero
(\cite[Chapter 8, page 79]{ArTa67}).

There have been different notions of genus fields according to 
different Hilbert class field definitions. 
M. Rosen \cite{Ros87} gave a definition of Hilbert class fields of $K$,
fixing a nonempty finite set $S_{\infty}$ of prime divisors of $K$.
Using Rosen's definition of Hilbert class field, it is 
possible to give a proper concept
of genus fields along the lines of number fields.

R. Clement \cite{Cle92} found a narrow genus
field of a cyclic extension of $k=\F(T)$ of prime degree $l$ dividing
$q-1$. She used the concept of Hilbert class field similar to that of a quadratic
number field $K$: it is the finite abelian extension of $K$ such
that the prime ideals of the ring of integers ${\mathcal O}_K$
of $K$ splitting there are precisely the principal ideals generated by 
an element whose norm is an $l$--power. 
S. Bae and J. K. Koo \cite{BaKo96}
were able to generalize the results of Clement with the methods
developed by Fr\"ohlich \cite{Fro83}. They defined the narrow genus
field for general global function fields and developed the analogue
of the classical genus theory. 
B. Angl\`es and J.-F. Jaulent
\cite{AnJa2000} used narrow $S$--class groups to establish
the fundamental results, using class field theory, 
for the genus theory of finite extensions of
global fields, where $S$ is a finite set of places.

G. Peng \cite{Pen2003} explicitly described the genus theory for
Kummer extensions $K$ of $k:=\F(T)$ of prime degree $l$,
based on the global function field analogue of the P. E.
Conner and J. Hurrelbrink exact hexagon.
C. Wittman \cite{Wit2007} extended Peng's results to the
case $l\nmid q(q-1)$ and used his results to study the $l$--part
of the ideal class groups of cyclic extensions of prime 
degree $l$ of $k$.
S. Hu and Y. Li \cite{HuLi2010} described explicitly the genus
field of an Artin--Schreier extension of $k$.

In \cite{MaRzVi2013, MaRzVi2015} it was
developed a theory of genus fields of
congruence function fields using Rosen's definition of Hilbert class field.
The methods used there 
were based on the ideas of Leopoldt using Dirichlet
characters and it was given a general description of $\g K$ in terms of
Dirichlet characters. The genus field 
$\g K$ was obtained for an abelian extension $K$ of
$k$. The method was used to give $\g K$ explicitly when
$K/k$ is a cyclic extension of prime degree $l\mid q-1$ (Kummer)
or $l=p$ where $p$ is the characteristic (Artin--Schreier) and
also when $K/k$ is a $p$--cyclic extension (Witt). Later
on, the method was used in \cite{BaRzVi2013} to describe
$\g K$ explicitly when $K/k$ is a cyclic extension of degree
$l^n$, where $l$ is a prime number and $l^n\mid q-1$.

In this paper we consider a finite abelian extension $K/k$.
We find the genus field of $K$ with respect to $k$.
Special consideration is given to
the genus field of a finite abelian $p$--extension 
of $k$, where $p$ is the characteristic.

The study of elementary abelian $p$--extensions, and more
generally abelian $p$--extensions, has been
considered by numerous authors. These extensions
appear in several contexts. In \cite{Ore33} O. Ore considered additive
polynomials using composition as multiplication. 
With this operation these polynomials are known as
{\em twisted polynomials} and this is one of the bases for 
{\em Drinfeld modules}. G. Lachaud \cite{Lau91} obtained 
an analogue of the Carlitz--Uchiyama bound for geometric
BCH codes and some consequences for cyclic codes. His
results are part of the analysis of the $L$--function of Artin--Schreier
extensions. Garcia and Stichtenoth \cite{GarSti91} studied 
field extensions $L/K$ given by an equation of the type
$y^q-y=f(x)\in K(x)$ where $q$ is a power of $p$ and $\F
\subseteq K$. Using a result of E. Kani \cite{Kan85} they obtained
a formula relating the genus of the extension and the genus
of the several subextensions of degree $p$. There are many fields
of this kind having the maximum number of rational places
allowed by Weil's bound, but they proved that fixed $K$,
this number of rational places is asymptotically bad. They also
used these extensions to find a family of fields whose
Weierstrass gap sequences are nonclassical.

In \cite{BaJaRzVi2016} we considered an additive 
polynomial $f(X)$ whose roots
belong to the base field and we proved results analogous to
the ones obtained by Garcia and Stichtenoth. 
More generally, we studied abelian extensions of type 
$C^n_{p^m}$, where $C_j$ denotes a cyclic group of 
order $j$, and such that the base field contains the finite
field $\F$, with $q=p^n$. 
For instance, given an additive polynomial $f(X)$, we have that
if the roots of $f$ are in the base field, any elementary
abelian $p$--extension can be obtained by means of an equation of
the type $f(X)=u$.
Furthermore, all the subextensions of degree $p$ over the
base field can be deduced from the equation $f(X)=u$.

We have studied genus fields in \cite{MaRzVi2013, MaRzVi2015,
MaRzVi2016}. The general result we present here goes along
the lines of the proof we presented in \cite{MaRzVi2013}, but it is much
simpler since now we consider in just one step the tame and the
wild ramification of the infinite prime. 
In \cite{MaRzVi2013} we first studied the case of tame ramification of
the infinite primes and next the general case. It turns out that it is possible
to consider the general case in just one step and in fact this approach
gives the genus field much faster and, 
in a way, more transparent. Furthermore, in
\cite{MaRzVi2013} we restricted ourselves to geometric extensions. Here we 
consider general finite abelian extensions, not necessarily geometric.

 We use this approach to study finite abelian
 $p$--extensions of $k$. Obtaining the genus field of this
family of extensions is much more transparent than the way it was
obtained in \cite{MaRzVi2013}. Our first main result is Theorem \ref{T2.1}.
As a corollary we obtain the general description of
the genus field of abelian $p$--extensions
in Theorem \ref{T2.2}.

Our second main result is the description of
what we call the {\em conductor of
constants} of an abelian extension $K/k$.
The classical Kronecker--Weber Theorem establishes that every
finite abe\-lian extension of ${\ma Q}$, the field of rational numbers, is 
contained in a cyclotomic field. Equivalently, the maximal abelian
extension of ${\ma Q}$ is the union of all cyclotomic
fields. In 1974, D. Hayes \cite{Hay74},
proved the analogous result for rational congruence function fields.
Hayes proved that
the maximal abelian extension of $k$ is the composite of
three linearly disjoint fields: the first one is
the union of all cyclotomic function fields; the second one is
the union of all constant extensions
and the third one is the union of all the subfields of the
corresponding cyclotomic function fields, where the
infinite prime is totally wildly ramified.

Given a finite abelian extension $K/k$, by the Kronecker--Weber
Theorem, using the notations of Section \ref{S2},  we have
$K\subseteq
{_n\cicl N{}}_m$ for some $n,m\in{\ma N}$ and $N\in R_T$.
The minimum $N$ and $n$ can be found by class field theory
by means of the conductor related to the finite primes and the
infinite prime respectively. However $m$ does not belong to
this category. In this paper we define the {\em conductor
of constants} as the minimum $m$ satisfying this
condition and describe $m$ in terms of some other invariants
of the extension. This is given in Theorems \ref{T2.1.AA} 
and \ref{T2.5.AA}.

The third main result is the explicit description of 
genus fields of finite abelian $p$--extensions of rational
function fields in case
we have enough constants. This is Theorem \ref{T3.1}.

To describe the genus fields of finite abelian $p$--extensions
of rational function fields without enough constants, we first prove
a result on the genus field of a composite of finite abelian extensions
of degree relatively prime to the order of the multiplicative
group of the field of constants, 
which shows that the genus field of the composite
is the composite of the
respective genus fields. The description of the genus field
of an arbitrary finite abelian extension of a global
rational function field of degree relatively prime to
the order of the multiplicative group of the field of
constants is the final main result, Theorem \ref{T3.10}.

\section{The genus field}\label{S2}

We will use the following notation. Let $k=k_0(T)$
be a global rational function field
of characteristic $p$, where $k_0={\ma F}_q$. Let $R_T=
{\ma F}_q[T]$ be the polynomial ring. Let $R_T^+$
denote the set of all monic irreducible polynomials in $R_T$.
For $N\in R_T$, $k(\Lambda_N)$ denotes the $N$--th Carlitz
cyclotomic function field. Let $\p$ be the pole of
the principal divisor $(T)$ in $k$,
which we call the {\em infinite prime}. The maximal real subfield $k(
\Lambda_N)^+$ of $k(\Lambda_N)$ is the decomposition
field of the infinite prime. For any field $L$ such that $k\subseteq
L\subseteq k(\Lambda_N)$, the real subfield $L^+$ of $L$
is $L^+:=k(\Lambda_N)^+\cap L$. The general results on cyclotomic
function fields can be consulted in \cite[Chapter 12]{Vil2006}.
Let $K/k$
be a finite abelian extension. From the Kronecker--Weber Theorem, we have
that there exist $n,m\in{\ma N}$ and $N\in R_T$ such that
\[
K\subseteq \ _nk(\Lambda_N)_m:= L_n k(\Lambda_N){\ma F}_{q^m},
\]
where $L_n$ denotes the subfield of $k(\Lambda_{1/T^{n+1}})$ of
degree $q^n$ and $k_m:={\ma F}_{q^m}(T)$ is the extension of
constants of $k$ of degree $m$. We have that $\p$ is totally
and wildly ramified in $L_n/k$. 
We also have that $\p$ is totally inert in $k_m/k$.

For any finite abelian extension $F$ of $k$, $\S F$ denotes the set of
prime divisors of $F$ above $\p$. 
For any finite abelian field extension $E/F$, let
$\e {E/F}$, $\f{E/F}$ and $h_{\infty}(E/F)$ denote the ramification 
index, the inertia degree and the decomposition number
of $\S F$ in $E$ respectively. For $P\in R_T^+$, $e_P(E/F)$ denotes
the ramification index of any prime in $F$ above $P$ in $E/F$. 
For any extension $F/k$, let $\g F$
denote the genus field of $F$ over $k$ as presented in the
introduction with $S=\S F$. When $F/k$ is a finite abelian extension, $\g F$ is the
maximal abelian extension contained in the Hilbert class field of $F$. 
The symbol $C_d$ will denote the cyclic group of $d$ elements.

For any field $F$, $W_v(F)$ denotes the ring of Witt vectors of
length $v$. The Witt operations will be denoted by $\Witt +$
and $\Witt -$.

Let $M:=L_nk_m$. Then 
\begin{gather}\label{Eq1}
\e{M/k}=q^n, \quad \f{M/k}=m  \quad \text{and} \quad h_{\infty}(M/k)=1.
\end{gather}
We have $M\cap k(\Lambda_N)=k$. The general results on genus fields
needed along this paper, can be found in \cite{MaRzVi2013, MaRzVi2015}.

First, we present a new proof of the fact that if $K\subseteq \cicl N{}$,
then $\g K\subseteq \cicl N{}$.

\begin{theorem}\label{T2.1-1} Let $k\subseteq  K\subseteq
k(\Lambda_N)$ for some $N\in R_T^+$. Then $\g K
\subseteq k(\Lambda_N)$. Furthemore, if the group of
Dirichlet characters of $K$ is $X$ and if $L$ 
is the field associated to $Y=\prod_{P\in R_T^+}
X_P$, then 
\[
\g K= K L^+.
\]
\end{theorem}

\begin{proof} Let $F/ K$ be an unramified abelian 
extension so that the elements of $\S K$ are fully
decomposed in $F/K$. In particular $\p$ is tamely
ramified.

By the Kronecker--Weber theorem, we have
$F\subseteq K(\Lambda_M)_m$ for some 
$M\in R_T^+$, $m\in {\ma N}$.

Let ${\mc I}$ be the inertia group of $\S K$ in $k(\Lambda_M)/k$
and let $B=k(\Lambda_M)^{\mc I}$.

\[
\xymatrix{k(\Lambda_M)\ar@{-}[dd]_{\mc I}
\ar@{-}[rr]&&k(\Lambda_M)_m
\ar@{-}[dd]\\ &F\ar@{-}[ddl] |!{[dl];[dr]}\hole\\  
B\ar@{-}[d]\ar@{-}[rr]&&B_m\ar@{-}[d]\\
 K\ar@{-}[rr]\ar@{-}[d]&& K_m\ar@{-}[d]\\
k\ar@{-}[rr]&&k_m}
\]

Since the elements of $\S B$ are of degree $1$, they are fully
inert in $B_m/B$. Furthermore, the elements of $\S B$ are fully
ramified in $k(\Lambda_M)/B$. Now, the elements of $\S K$ are
fully decomposed in $B/K$ so we obtain that $B$ is the
decomposition field of $\S K$ in $k(\Lambda_M)_m/K$.
It follows that $F\subseteq B\subseteq k(\Lambda_M)$.

Let $Z$ be the group of Dirichlet characters associated to
$F$. Since $F/K$ is unramified, it follows that
$X\subseteq Z\subseteq Y$, that is, $F\subseteq L$
since $L$ is the maximal abelian extension contained
in some cyclotomic function field such that $L/K$ is
unramified in the finite primes. In particular, we may
take $M=N$. Therefore $\g K=L^{\mc D}$ where
${\mc D}$ is the decomposition group of $\S K$ in $L/K$.

Now, $\S K$ decompose fully in $K L^+/ K$ since $\p$
decomposes fully in $L^+/k$. Since $L/K$ is unramified,
we have $K L^+\subseteq L$ so that $K L^+/ K$ is
unramified. Hence $K L^+\subseteq \g K$ and we 
obtain that $K L^+\subseteq \g K\subseteq L$.

Finally, let us see that $\S{K L^+}$ is fully ramified in the
extension $L/K L^+$. In fact this follows from the fact that
$L^+\subseteq K L^+ \subseteq L$ and from that 
$\S{L^+}$ is totally ramified in $L/L^+$.
Since $K L^+\subseteq \g K\subseteq L$ and
$\g K/K L^+$ is unramified, it follows that $\g K=K L^+
\subseteq k(\Lambda_N)$. 
\end{proof}

Our first main result is

\begin{theorem}\label{T2.1} With the above notations, let $K/k$
be a finite abelian extension. Let
\[
E:=KM\cap k(\Lambda_N).
\]

Then 
\[
\g K=\g {E^{H_1}} K= (\g E K)^H,
\]
where $H$ is the decomposition group of any prime in $\S K$
in $\g E K/K$, $H_1:= H|_{\g E}$ and $H_2:=H_1|_E$.

Let $d:= \f{EK/K}$. We have $H\cong H_1\cong H_2\cong
C_d$ and $d|q-1$. We also have $\g E K/\g K$ and $EK
/E^{H_2}K$ are extensions of constants of degree $d$.
Finally, the field of constants of $\g K$ is ${\ma F}_{q^t}$, 
where $t$ is the degree of $\S K$ in $K$.
\end{theorem}

\begin{proof}
The proof that the field of constants of $\g K$ is ${\ma F}_{q^t}$
is the same as the one in \cite[Lemma 4.1]{MaRzVi2013}.
We repeat the argument for the sake of completeness. Let
$K_r$ be the extension of constants of $K$ of degree $r$.
Since the degree of any element of $\S K$ is $t$, the elements
of $\S K$ decompose into $\gcd(t,r)$ elements of $K_r$. Therefore
the elements of $\S K$ decompose fully if and only if
$\gcd(t,r)=r$ if and only if $r|t$. The assertion follows.

Since $k(\Lambda_N)\cap M=k$ and $E=KM\cap k(\Lambda_N)$,
from the Galois correspondence, between $k(\Lambda_N)/k$
and $k(\Lambda_N)M/M$, $E$ corresponds to $KM$. Hence
$KM=EM$ corresponds to $E$. Thus
\begin{gather*}
KM=EM.
\end{gather*}

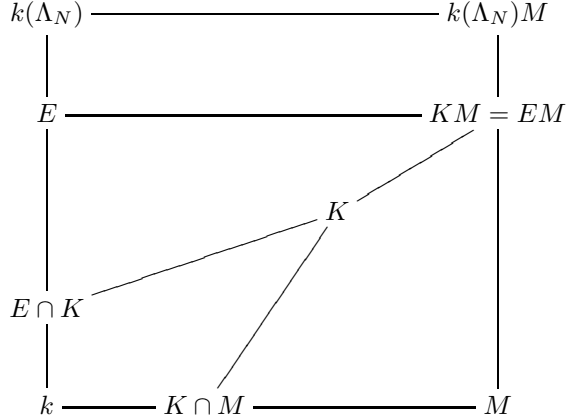
\[
\xymatrix{
k(\Lambda_N)\ar@{-}[d]\ar@{-}[rrr]&&&k(\Lambda_N)M\ar@{-}[d]\\
E\ar@{-}[rrr]\ar@{-}[dd]&&&KM=EM\ar@{-}[ddd]\\
&&K\ar@{-}[ru]\ar@{-}[ddl]\\
E\cap K\ar@{-}[rru]\ar@{-}[d]\\
k\ar@{-}[r]&K\cap M\ar@{-}[rr]&&M
}
\]

Now $E\cap K\subseteq \g E\cap K\subseteq k(\Lambda_N)\cap K
= (KM\cap k(\Lambda_N))\cap k(\Lambda_N)\cap K=
E\cap k(\Lambda_N)\cap K=E\cap K$. Therefore
\begin{gather*}
E\cap K=\g E\cap K=k(\Lambda_N)\cap K.
\end{gather*}

We have
$[E:k]=[EM:M]= [KM:M]=[K:K\cap M]$. Thus
\begin{gather}\label{Eq4}
[K:k]=[E:k][K\cap M:k].
\end{gather}

Next, we will prove that $EK/K$ is unramified. First note
that $E\subseteq EK\subseteq EKM=E\cdot EM=EM$. In 
the extension $M/k$, $\p$
is the only ramified prime. Hence in $KM/E$ the only 
possible ramified primes are those in $\S E$. We also have
that in the extension $KM/K$ the only possible ramified primes are
the elements of $\S K$ and since $K\subseteq EK\subseteq
EM=KM$, the only possible ramified primes in $EK/K$
are those in $\S K$.
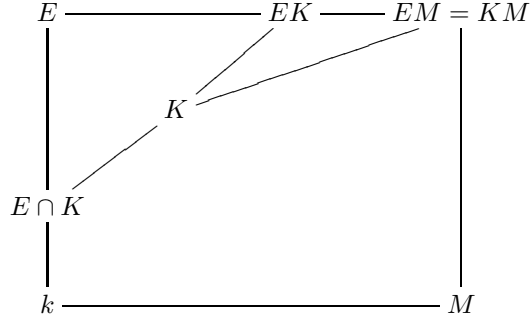
\[
\xymatrix{
E\ar@{-}[rr]\ar@{-}[dd]&&EK\ar@{-}[dl]\ar@{-}[r]&EM=KM
\ar@{-}[dll]\ar@{-}[ddd]\\
&K\ar@{-}[dl]\\
E\cap K\ar@{-}[d]\\
k\ar@{-}[rrr]&&&M
}
\]

From (\ref{Eq1}) we have
\[
e_{\infty}(EK/K)\mid e_{\infty}(M/K\cap M)\quad \text{and}\quad
e_{\infty}(M/K\cap M)\mid e_{\infty}(M/k)=q^n.
\]

On the other hand, we have
\begin{gather*}
e_{\infty}(EK/K)\mid e_{\infty}(E/E\cap K) \quad \text{and} \quad
e_{\infty}(E/E\cap K)\mid e_{\infty}(k(\Lambda_N)/k)=q-1.
\intertext{Thus}
e_{\infty}(EK/K)\mid \gcd(q^n,q-1)=1
\end{gather*}
and $EK/K$ is unramified.

Now, we have that
\[
e_{\infty}(EK/K)f_{\infty}(EK/K)\mid e_{\infty}(E/E\cap K)
f_{\infty}(E/E\cap K),
\]
and $\e{EK/K}=1$, $\f{E/E\cap K}=1$. Therefore
$f_{\infty}(EK/K)\mid e_{\infty}(E/E\cap K)$ and $e_{\infty}(
E/E\cap K)\mid q-1$. Thus $f_{\infty}(EK/K)\mid q-1$.

Therefore we have that $EK/K$ is unramified,
the inertia degree of $\S K$ in $EK/K$ is 
$d=\f{EK/K}$ and $d\mid q-1$.
Since $\g E/E$ is unramified and $\S E$ decomposes fully
in $\g E/E$, the same holds in $\g EK/EK$. In this way we
obtain that $\g E K/K$ is an unramified extension and the 
inertia degree of $\S K$ is $d$.

Recall that $H$ is the decomposition group of any prime
in $\S K$ in $\g E K/K$ and
let $H_1:=H|_{\g E}$. Observe that
$|H|=d$. Since $\g E\cap K=E\cap K$, from the
Galois correspondence we obtain that $H\cong H_1$, $|H|=|H_1|$
and $\g E^{H_1} K=(\g E K)^H$. Analogously, 
$H_2\cong H_1$. Furthermore, $H_1\subseteq I_{\infty}(
k(\Lambda_N)/k)\cong C_{q-1}$, where $I_{\infty}$ denotes
the inertia group of $\p$. Therefore $H$ is a cyclic group, $H\cong
H_1\cong H_2\cong C_d$.

Since $\S K$ decomposes fully in $\g E^{H_1}K/K$, it follows that
\begin{gather*}
\g E^{H_1} K\subseteq \g K.
\end{gather*}

Let $E_1:=E \g E^{H_1}\subseteq \g E$. Now $H_1\subseteq
I_{\infty}(E/E\cap K)$, so $\S {\g E^{H_1}}$ is fully ramified in 
$\g E/\g E^{H_1}$. Therefore $\S {E_1}$ is fully ramified in 
$\g E/E_1$. On the other hand $\S E$ decomposes fully in $\g E/E$.
Hence $\S {E_1}$ decomposes fully in $\g E/E_1$. That is, $\S {E_1}$
ramifies and decomposes fully in $\g E/E_1$. Therefore
\begin{gather*}
\g E=E_1=E\g E^{H_1}.
\intertext{It follows that}
(\g E K)^H=\g E^{H_1} K\subseteq \g K \quad \text{and} \quad 
E\g E^{H_1}=\g E.
\end{gather*}

To prove the other containment, we define $C:=\g K M\cap k(\Lambda_N)$.
We have
\begin{gather*}
E\subseteq EM=KM\subseteq \g K M,\quad E\subseteq k(\Lambda_N).
\intertext{Therefore}
E\subseteq \g KM\cap k(\Lambda_N)=C,\quad \text{that is}\quad E\subseteq C.
\end{gather*}

Furthermore, $\g E^{H_1}\subseteq \g E^{H_1} K
\subseteq \g K\subseteq \g K M$ and $\g E^{H_1}\subseteq \g E
\subseteq k(\Lambda_N)$. Thus $\g E^{H_1}
\subseteq \g K M\cap k(\Lambda_N)=C$.
Hence $\g E^{H_1}\subseteq C$. Therefore
\begin{gather}\label{Eq6}
\g E =E \g E^{H_1} \subseteq C.
\end{gather}

\begin{tiny}
\begin{gather*}
\xymatrix{
&k(\Lambda_N)\ar@{-}[rrrr]\ar@{-}[d]&&&&k(\Lambda_N)M\ar@{-}[d]\\
&C\ar@{--}[dd]\ar@{-}[rrrr]&&&&CM=
\g K M\ar@{--}[dd]\ar@{-}[dl]\ar@/^2pc/@{-}[dddd]^{
\text{unramified}}\\
&&&&\g K\ar@{-}[ddl]
|!{[ld];[ddd]}\hole
\ar@/^2pc/@{-}[ddddl]^{\text{unramified}}
|!{[dl];[dr]}\hole
|!{[dl];[ddd]}\hole
|!{[dddll];[ddd]}\hole\\
&\g E=\g E^{H_1} E\ar@{-}[ddr]
\ar@{-}[rr]\ar@{-}[dl]_{H_1=H|_{\g E}}&&
\g E K\ar@{-}[d]_H\ar@{-}[rr]
|!{[ur];[d]}\hole
\ar@{-}[ddr]&&\g E M\ar@{-}[dd]\\
\g E^{H_1}\ar@{-}[rrr]
|!{[ru];[rrd]}\hole
\ar@{-}[rrrruu]
|!{[ru];[rrd]}\hole
|!{[ru];[rru]}\hole
\ar@{-}[ddr]&&&\g E^{H_1}K=(\g E K)^H\ar@{-}[dd]\\
&&E\ar@{-}[rr]
|!{[ru];[rd]}\hole
\ar@{-}[dl]&&EK\ar@{-}[r]\ar@{-}[dl]&EKM=EM=
KM\ar@{-}[lld]\ar@{-}[dd]\\
&E\cap K\ar@{-}[rr]\ar@{-}[d]&&K\\
&k=k(\Lambda_N)\cap M\ar@{-}[rrrr]^{e_{\infty}=q^n,\quad f_{\infty}=m}&&&&M
}
\end{gather*}
\end{tiny}

Since $C=\g KM\cap k(\Lambda_N)$, from the Galois correspondence we 
have $CM=\g K M$. Now, since $\g K/K$ is unramified and $\S K$
decomposes fully, it follows that 
\begin{gather}\label{Eq7}
CM/KM\quad \text{is unramified and $\S {KM}$
decomposes fully}.
\end{gather}

We now prove that $C/E$ is unramified. From (\ref{Eq7})
follows that $CM/KM$ is
unramified. Now, in $KM=EM$ over $E$, the only ramified primes are those in
$\S E$ and they have ramification index equal to $q^n$. 
It follows that the only ramified primes in $CM/E$
are those in $\S E$. Hence the only possible ramified primes in
$C/E$ are those in $\S E$. Now
\begin{gather*}
e_{\infty}(C/E)\mid e_{\infty}(CM/E)=q^n
\quad\text{and}\quad e_{\infty}(C/E)\mid
e_{\infty}(k(\Lambda_N)/k)=q-1
\intertext{so that}
e_{\infty}(C/E)\mid \gcd(q^n,q-1)=1.
\end{gather*}
Therefore $C/E$ is an unramified extension.

On the other hand, being $\S E$ unramified in $C/E$, $\S E$ decomposes
fully in $C/E$ since $C\subseteq k(\Lambda_N)$. It follows that
$C\subseteq \g E$. From this and equation (\ref{Eq6}), we obtain
\begin{gather*}
C=\g E \quad\text{and}\quad \g E M=CM=\g KM.
\end{gather*}

We have $\g E K\subseteq \g E \g K$. Since $\g K/K$ is unramified and
$\S K$ decomposes fully in $\g K$, the same holds in the
extension $\g E\g K/\g E K$. In particular $h_{\infty}(\g E\g K/\g E K)
=[\g E\g K:\g E K]$.

Now, in the extension $\g EM/\g E$, the only ramified primes are those in
$\S {\g E}$ and we have $\e {\g EM/\g E}=q^n$ and $\f{\g EM/\g E}=m$ because
$\e {\g E/k}\mid q-1$ which is relatively prime to $q$, $\f{\g E/k}=1$,
$\e {M/k}=q^n$ and $\f {M/k}=m$.

\[
\xymatrix{
\g E\ar@{-}[rrr]\ar@{-}[d]&&&\g E M\ar@{-}[d]\\
\g E\cap M=k\ar@{-}[rrr]_{e_{\infty}=q^n, f_{\infty}=m}&&& M
}
\]

Let $F_1$ and $F_2$ two fields such that $k\subseteq F_1\subseteq F_2
\subseteq M$. Let $R_i=\g E F_i$, $i=1,2$. Since $\f{\g E/k}=1$
and $\e{\g E/k}\mid q-1$, it follows from the Galois correspondence
between $M/k$ and $\g E M/\g E$ that $\e{R_i/\g E}=\e {F_i/k}$ and that
$\f{R_i/\g E}=\f {F_i/k}$, $i=1,2$. Therefore $\e{F_2/F_1}=\e{R_2/R_1}$
and $\f{F_2/F_1}=\f{R_2/R_1}$.

Since $h_{\infty}(M/k)=1$, we have $h_{\infty}(R_2/R_1)=1$. In particular
\begin{gather}
R_1\neq R_2\iff F_1\neq F_2\iff \e{F_2/F_1}>1\text{\ or\ }\f{F_2/F_1}>1\nonumber\\
\iff \e{R_2/R_1}>1\text{\ or\ }\f{R_2/R_1}>1. \label{Eq4.5'}
\end{gather}

Since 
\[
\g E\subseteq \g E K\subseteq \g E \g K \subseteq \g K M=\g E M,
\]
$\S {\g E K}$ is unramified in $\g E\g K/\g E K$ and $\S{\g E K}$
decomposes fully, we obtain that $\e{\g E\g K/\g E K}=1$ and 
$\f{\g E\g K/\g E K}=1$. From (\ref{Eq4.5'}), it follows that
\[
\g E\g K=\g E K.
\]
Therefore $\g K\subseteq \g E\g K=\g E K$. Since $\g E K/K$ is unramified,
if ${\mc D}$ is the decomposition group of $\S K$ in $\g E K/K$, we
obtain that $\g K=
(\g E K)^{\mc D}$. Finally, we have
\[
\f{\g E K/K}=\f{\g E K/EK}\f{EK/K}=1\cdot d=d.
\]
Hence ${\mc D}=H$ and $\g K=(\g E K)^{\mc D}=(\g E K)^H=\g E^{H_1} K$. 

Finally, it remains to show that $\g E K/\g K$ and $EK/E^{H_2}K$ are extensions of
constants. 

Since $\g K M=\g E M$ and $\g E\g K=\g E K$, we have
\begin{gather*}
\g K=(\g E K)^H\subseteq \g E K\subseteq \g E\g K\subseteq \g E\g KM= \g E M.
\end{gather*}

Set $F_1=\g K\cap M$ and $F_2=\g E K\cap M$. We have
$d=[\g EK:\g K]=\f{\g E K/\g K}=[F_2:F_1]=\e{F_2/F_1}\f{F_2/F_1}
h_{\infty}(F_2/F_1)$. Since $\e{F_2/F_1}\mid q^n$ 
and $h_{\infty}(F_2/F_1)=1$,
it follows that
\begin{gather*}
\e {F_2/F_1}=\e{\g EK/\g K}=1\quad\text{and}\quad
\f {F_2/F_1}=\f{\g EK/\g K}=d.
\end{gather*}
Therefore $k\subseteq F_1\subseteq F_2\subseteq M$ and $\e{F_2/F_1}=1$.

Let $a$ and $b$ be such that $F_2\subseteq F_1 k_bL_a$.
Let $A_i=F_i k_b\cap L_a$, $i=1,2$. Note that because $\e{F_2/F_1}=1$
and $F_i k_b=A_i k_b/A_i$, $i=1,2$, are extensions of constants,
we have $\e{A_2/A_1}=1$.
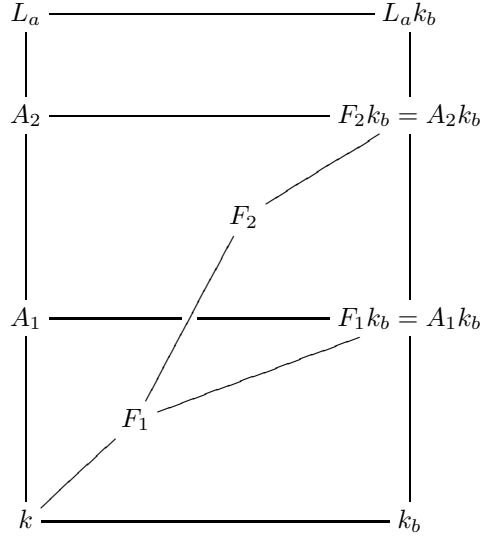
\begin{gather*}
\xymatrix{
L_a\ar@{-}[rrr]\ar@{-}[d]&&& L_a k_b\ar@{-}[d]\\
A_2\ar@{-}[rrr]\ar@{-}[dd]&&& F_2 k_b=A_2 k_b\ar@{-}[dd]\\
&&F_2\ar@{-}[ur]\ar@{-}[ddl]\\
A_1\ar@{-}[rrr]
|!{[rru];[rd]}\hole
\ar@{-}[dd]&&& F_1 k_b=A_1 k_b\ar@{-}[dd]\\
&F_1\ar@{-}[rru]\ar@{-}[dl]\\
k\ar@{-}[rrr]&&&k_b
}
\\
\e{F_2k_b/F_1k_b}=\e{F_2/F_1}=\e{A_2/A_1}=1.
\end{gather*}

Since $L_a/k$ is totally ramified at $\p$, it follows that
$A_1=A_2$. Therefore $F_2 k_b=F_1k_b$ and $F_2/F_1$
is an extension of constants. 

Recall $F_1=\g K\cap M$. We consider
$\g K\subseteq \g E K\subseteq \g K M=\g E M$:
\[
\xymatrix{
\g K\ar@{-}[r]\ar@{-}[d]&\g E  K\ar@{-}[r]\ar@{-}[d]&
 K M=\g EM\ar@{-}[d]\\
F_1\ar@{-}[r]&F_2\ar@{-}[r] &M}
\]
Therefore $\g K\subseteq F_2\g K=\g E K$.
It follows that $\g E  K/\g  K$ is an extension
of constants of degree $[\g EK:\g K]=|H|=d$. 

The proof that
$EK/E^{H_2}K$ is an extension of constants is completely similar.

This finishes the proof of the theorem.
\end{proof}

For the particular case of a finite abelian $p$--extension, we have
that, on the one hand, $d\mid q-1$ and, on the other hand, $d\mid
[EK:K]$. Since $K/k$ is a $p$--extension, we obtain from
(\ref{Eq4}), that $E/k$ is also a $p$--extension. Finally, since
$\Gal(EK/k)\to \Gal(E/k)\times 
\Gal(K/k)$, $\sigma\mapsto (\sigma|_{E},\sigma|_{K})$
is injective, it follows that $EK/k$ is also a $p$--extension.
Therefore $d\mid p^a$ for some $a$. Thus $d=1$.
We have proved

\begin{theorem}\label{T2.2} With the above notations, let $K/k$
be a finite abelian $p$--extension. Let
\[
E:=KM\cap k(\Lambda_N).
\]
Then $\g K=\g E K$ and $\g K/k$ is an abelian $p$--extension.
\end{theorem}

\begin{proof} The last assertion follows from the fact that
$\g E/k$ is also an abelian $p$--extension.
\end{proof}

\section{Conductor of constants}\label{S14.5.1}

Let $K$ be a finite abelian extension of $k$. By the Kronecker--Weber
we have that there exist $n,m\in{\ma N}$ and 
$N\in R_T$ such that $ K\subseteq {_nk(\Lambda_N)_m}$.
The minima $n$ and $N$ satisfying this condition are given
by class field theory by means of the local conductors of
the extension $K/k$: $n$ for $\p$ and $N$ for the finite primes.

In this section we will determine the minimum $m$ satisfying
the above condition and we will see that this $m$ is related
to the number $d$ given in Theorem \ref{T2.1}. The number
$m$ will be called {\em the conductor of constants} of the
abelian extension $K/k$.

First, let $n, m\in{\ma N}$ and $N\in R_T$ be such that 
$ K\subseteq {_nk(\Lambda_N)_m}$ and
where $m$ is the minimum with respect to this
condition. Note that $m$ might depend on
$n$ and $N$.  Consider the following
diagram of Galois extensions
\[
\xymatrix{
{_n k(\Lambda_N)}\ar@{-}[rr]\ar@{-}[dd]&&U= {_nk(\Lambda_N)} K
\ar@{-}[r]\ar@{-}[dl]\ar@{--}[dd] & {_nk(\Lambda_N)_m}\ar@{-}[dd]\\
& K\ar@{-}[dl] \\ k\ar@{-}[rr] && k_{m'}\ar@{-}[r] &k_m
}
\]
That is, let $U:={_nk(\Lambda_N)}  K$ and $k_{m'}:=U\cap k_m$.
From the Galois correspondence, we have that $U
={_nk(\Lambda_N)}  K={_nk(\Lambda_N)}
k_{m'}={_nk(\Lambda_N)_{m'}}\supseteq  K$.

Since $m$ is minimal, we obtain that $m'=m$. That is,
$m$ is determined by the equality
\begin{gather}\label{conductor de constantes}
{_nk(\Lambda_N)} K={_nk(\Lambda_N)_m}.
\end{gather}

Now, we will see that $m$ is independent of $n$ and of $N$.
Let $n_i\in{\ma N}$, $N_i\in R_T$ and $m_i\in{\ma N}$ be the minimum
such that $ K\subseteq {_{n_i} k(\Lambda_{N_i})_{m_i}}$,
$i=1,2$.

Let $n_0:=\max \{n_1,n_2\}$, $N_0=\lcm [N_1,N_2]$ and $m_0\in
{\ma N}$ be minimum such that $K\subseteq {_{n_0}k(\Lambda_{N_0})_{m_0}}$.
From (\ref{conductor de constantes}), it follows that
\begin{align*}
{_{n_0}k(\Lambda_{N_0})} K&= L_{n_0}\big({_{n_i}k(\Lambda_{N_i})}
k(\Lambda_{N_0})\big) K=L_{n_0}\big({_{n_i}k(\Lambda_{N_i})} K\big)
k(\Lambda_{N_0})\\
&=L_{n_0} \big({_{n_i}k(\Lambda_{N_i})_{m_i}} k(
\Lambda_{N_0})\big)={_{n_0}k(\Lambda_{N_0})_{m_i}},\quad\text{and}\\
{_{n_0}k(\Lambda_{N_0})} K&={_{n_0}k(\Lambda_{N_0})_{m_0}}.
\end{align*}
Therefore $m_1=m_2=m_0$.

So, we consider $K\subseteq {_n\cicl N{}_m}$ with
$m$ the minimum. Let
$F:=  K\cap {_n\cicl N{}}$ and consider the following Galois square
(see (\ref{conductor de constantes}))
\[
\xymatrix{
{_n\cicl N{}}\ar@{-}[rr]^{m\phantom{xxxxx}}\ar@{-}[d] && {_n\cicl N{}_m}
 ={_n\cicl N{}} K
\ar@{-}[d]\\ F\ar@{-}[rr]^m\ar@{-}[d]&& K \ar@{-}[dll]\\ k
}
\]

Let $t$ be the degree of $\S  K$ in $ K$. That is, $t=f_{\infty}(K/k)$. We have
\begin{gather*}
e_{\infty}({_n\cicl N{}_m}/{_n\cicl N{}})=1,\quad 
f_{\infty}({_n\cicl N{}_m}/{_n\cicl N{}})=m.
\intertext{In particular}
\{1\}=I_{\infty}({_n\cicl N{}_m}/{_n\cicl N{}})\subseteq I_{\infty}( K/F),\\
C_m\cong D_{\infty}({_n\cicl N{}_m}/{_n\cicl N{}})\subseteq
D_{\infty}( K/F).
\end{gather*}

Since $[ K:F]=m$ and $m\leq |D_{\infty}( K/F)|\leq [ K:F]=m$, 
it follows that $|D_{\infty}( K/F)|=m$ and that $D_{\infty}( K/F)\cong C_m$. In
particular we have
$h_{\infty}( K/F)=1$ and $h_{\infty}({_n\cicl N{}_m}/{_n\cicl N{}})=1$.

On the other hand, we have
\begin{gather*}
t=f_{\infty}( K/k)=f_{\infty}( K/F)f_{\infty}(F/k)=\f{K/F}\cdot 1=f_{\infty}( K/F),
\intertext{that is, $f_{\infty}( K/F)=t$. Furthermore}
e_{\infty}( K/F)f_{\infty}( K/F)h_{\infty}( K/F)=\e{K/F}\cdot t\cdot 1=m,
\end{gather*}
so that $e_{\infty}( K/F)=\frac{m}{t}$. Hence
\begin{gather}\label{Eq3.2}
m=[K:F]=\f{K/F}\e{K/F}=t \e{K/F}=t\frac{\e{K/k}}{\e{F/k}}.
\end{gather}

Now we shall investigate the relation between $m$ 
and $d=\f{\g E K/\g K}$ given in
Theorem \ref{T2.1}. Recall that $M=L_nk_m$,
$E=KM\cap \cicl N{}$ and that $EM=KM$.
We have 
\[
\g E\subseteq \g E K\subseteq \g E K L_n\subseteq \g E K M =\g E EM=\g E M.
\]

Let $A:=\g EK\cap M$ and $B:=\g E K L_n \cap M$. From the Galois correspondence
we have $\g E K=\g E A$ and $\g E K L_n=\g E B$. 
\[
\xymatrix{
\g E\ar@{-}[d]\ar@{-}[r]&\g E K\ar@{-}[d]\ar@{-}[r]&\g E KL_n\ar@{-}[d]
\ar@{-}[r]&\g E M\ar@{-}[d]\\
k\ar@{-}[r]&A\ar@{-}[r]&B\ar@{-}[r]&M
}
\]
We have $L_n\subseteq
\g E K L_n\cap M=B\subseteq M=L_nk_m$. Therefore $B/L_n$ is an
extension of constants. Say $B=L_n k_{m'}$ with $m'|m$. From the
Galois correspondence, we obtain
\[
K\subseteq \g E  K L_n=\g E B=\g E L_n k_{m'}\subseteq 
\cicl N{} L_n k_{m'}={_n\cicl n{}}_{m'}.
\]
Since $m$ is the minimum, $m'=m$, $B=M$ and $\g E KL_n=\g EM$.

Now, $\g E(A L_n)=(\g E A)L_n=(\g E K)L_n=\g E M$. From the
Galois correspondence it follows that $A L_n=M$. We consider the
following Galois square:
\[
\xymatrix{
L_n\ar@{-}[d]\ar@{-}[r]&AL_n=M=L_nk_m\ar@{-}[d]\\
A\cap L_n\ar@{-}[r]&A
}
\]
We have $\f{AL_n/L_n}=\f{M/L_n}=m$ and $\e{AL_n/L_n}=
\e{M/L_n}=1$. Thus
\begin{align*}
\{1\}&=I_{\infty}(AL_n/L_n)\subseteq I_{\infty}(A/A\cap L_n)\quad\text{and}\\
C_m&\cong D_{\infty}(AL_n/L_n)\subseteq D_{\infty}(A/A\cap L_n).
\end{align*}

Because $[A:A\cap L_n]=[M:L_n]=m$, it follows that $D_{\infty}(A/A\cap L_n)
\cong C_m$, $\e{A/A\cap L_n}=1$ and $\f{A/A\cap L_n}=m$. Therefore
$\f{\g EK/k}=\f{\g EK/\g K}\f{\g K/K}\f{K/k}=d\cdot 1\cdot t= dt=td$. Thus
\begin{gather*}
\f{\g E M/\g E K}=\frac{\f{\g EM/k}}{\f{\g E K/k}}=\frac{m}{td}.
\intertext{Finally}
\frac{m}{td}=\f{\g EM/\g E K}|[\g EM:\g EK]=[M:A]=
[L_n:A\cap L_n]|[L_n:k]=q^n.
\intertext{It follows that}
m=td p^s
\end{gather*}
for some $s\in{\ma N}\cup \{0\}$.

Furthermore, $\f{K_m/K}=\frac{m}{t}=\e{K/F}$.
Note that 
\[
td=\f{K/k}\f{EK/K}=\f{EK/k}.
\]

We have obtained

\begin{theorem}[Conductor of constants 1]
\label{T2.1.AA} Let $ K$ be a finite abelian extension of
$k$. Let $n,m\in{\ma N}$ and $N\in R_T$ be such that
$K\subseteq {_n\cicl N{}_m}$ and such that $m$ is
minimum with this property. Then $m$ is independent of
$n$ and $N$. Let $t=f_{\infty}( K/k)$ be the degree
of the infinite primes of $K$. Let $M=L_nk_m$,
$E=KM\cap \cicl N{}$,
$F=K\cap {_n\cicl N{}}$ and 
$d=\f{EK/K}=\f{\g E K/\g K}$.
Then 
\begin{gather*}
{_n\cicl N{}} K={_n\cicl N{}}_m
\intertext{and}
m=[K:F]=t	\e{K/F}=td p^s=\f{EK/k} p^s
\end{gather*}
for some $s\geq 0$. In particular
$$
\e{K/F}=dp^s=\f{K_m/K}.
\eqno{\fin}
$$
\end{theorem}

\begin{remark}\label{R2.2AA}{\rm{
When $p\nmid \frac{m}{t}$, in particular when $K/k$
is tamely ramified at $\p$, we have $s=0$ and 
$m=td$. In the general case, we may have $s\geq 1$.
}}
\end{remark}

\begin{example}\label{Ex2.3AA}{\rm{
Let $p$ be any prime and let $q=p$. Let $X:=1/T$. We
have $L_1:=\cicl X2^{{\ma F}_p^{\ast}}$ and $[L_1:k]
=p$. We have that $L_1/k$ is an Artin--Schreier
extension. It is not necessary to give the explicit description
of $L_1$, however for the convenience of the reader we give 
a generator of $L_1$.
Let $\lambda$ be a generator of $\Lambda_{X^2}$
such that $\lambda^{p-1}$ is a generator of $\cicl X2^+=L_1$.
Now $\lambda$ is a root of the cyclotomic polynomial $\Psi_{X^2}
(u)$. We have that $\Psi_{X^2}(u)=\Psi_X(u^X)$ where $u^P$
denotes the Carlitz action. Since $\Psi_X(u)=u^P/u=u^{p-1}+X$,
it follows that $\Psi_{X^2}(\lambda)=(\lambda^p+X\lambda)^{p-1}+X$.
Set $\mu:=\lambda^{p-1}$ and $\xi:=\mu+X$. Then we obtain
\begin{gather*}
\xi^p-X\xi^{p-1}+X=0.
\intertext{Finally, if $\delta:=1/\xi$, then $L_1=k(\delta)$ with}
\delta^p-\delta=-1/X=-T,\quad \delta=\frac{T}{T\lambda^{p-1}+1}.
\end{gather*}

Let $\alpha$ be a solution of $y^p-y=1$. Then ${\ma F}_p(\alpha)=
{\ma F}_{p^p}$, $k_p={\ma F}_p(\alpha)(T)={\ma F}_{p^p}(T)$
and $L_1k_p=k(\alpha, \delta)$. The $p+1$ extensions $K/k$ of
degree $p$ over $k$ such that $k\subseteq K\subseteq L_1k_p$
are $\{k(\alpha+i\delta)\}_{i=0}^{p-1}$ and $L_1$. Set $K:=k(
\alpha+\delta)$. Then $K\neq k_p$ and $K\neq L_1$. Then
$K=k(z)$ with $z^p-z=1-T$.

Let $N\in R_T$ be arbitrary. Then $K\subseteq L_1k_p\subseteq
{_1\cicl N{}}_p$ and $K\nsubseteq {_1\cicl N{}}_1$. Therefore 
$m=p$ and $M=L_1k_p$. We have $\f{K/k}=1$, $\e{K/k}=p$. We also have 
$E:=KM\cap \cicl N{}=M\cap \cicl N{}=k$. Therefore $\g E=k$ and
$\g K=\g E K=K$. It follows that $EK=K$ and $\f{EK/K}=d=1$. Hence
$td=1\neq m=p$. In this example $s=1$.
}}
\end{example}

We will compute $m$ in another way. First, with the same proof 
as the one for Theorem \ref{T2.1} we obtain

\begin{theorem}\label{T2.4AA} Let $K/k$
be a finite abelian extension. Let
\[
R:=K_m\cap{_n\cicl N{}}.
\]

Then 
\[
\g K=\g {R^{{\mathcal H}_1}} K= (\g R K)^{\mathcal H},
\]
where ${\mathcal H}$ is the decomposition group of any prime in $\S K$
in $\g R K/K$, ${\mathcal H}_1:= {\mathcal H}|_{\g R}$ 
and ${\mathcal H}_2:={\mathcal H}_1|_R$.

Let $d^{\ast}:= \f{RK/K}$. We have ${\mathcal H}\cong 
{\mathcal H}_1\cong {\mathcal H}_2\cong
C_{d^{\ast}}$ and $d^{\ast}|q-1$. We also have $\g R K/\g K$ and $RK
/R^{{\mathcal H}_2}K$ are extensions of constants of degree $d^{\ast}$.
Finally, the field of constants of $\g K$ is ${\ma F}_{q^t}$, 
where $t$ is the degree of $\S K$ in $K$. $\fin$
\end{theorem}

Let now $F= K\cap {_n\cicl N{}}$ and consider the following Galois
squares
\begin{gather*}
\xymatrix{
{_n\cicl N{}}\ar@{-}[rr]\ar@{-}[d]&&{_n\cicl N{}_m}\ar@{-}[d]\\
R\ar@{-}[rr]\ar@{-}[dd]\ar@{-}[r]&& K_m=R_m\ar@{-}[dd]\ar@{-}[dl]\\
& K\ar@{-}[dl]\\ k\ar@{-}[rr]&&k_m
}
\\
\xymatrix{
{_n\cicl N{}} \ar@{-}[rr]\ar@{-}[d]&&{_n\cicl N{}} K={_n \cicl N{}_m}
\ar@{-}[d]\\ 
C\ar@{-}[rr]\ar@{-}[d]&&R_m=K_m\ar@{-}[d]\\
R= K_m\cap{_n\cicl N{}}\ar@{-}[rr]\ar@{-}[d]&&R K\ar@{-}[d]\\
F= K\cap {_n\cicl N{}}\ar@{-}[rr]&& K
}
\end{gather*}

Since $R= K_m\cap {_n\cicl N{}}$, it follows that $ K_m=R_m$.
Now, $ K,R\subseteq RK \subseteq  K_m=R_m$.

Let $C:= K_m\cap {_n\cicl N{}}$. Then $C=R$ and, from the
Galois correspondence, we have $R K=R_m= K_m$.

It follows that the field of constants of $R K$ is ${\ma F}_{q^m}$.
The field of constants of $R\g K$ is also ${\ma F}_{q^m}$.

Now, the field of constants of $\g K$ is ${\ma F}_{q^t}$.
On the other hand we have that
$R\g K/\g{R^{{\mathcal H}_1}} K=\g K$ is an extension of constants
of degree $d^{\ast}=|{\mathcal H}_1|$. Thus, the field of constants of 
$R\g K$ is ${\ma F}_{q^{td^{\ast}}}$. It follows that $td^{\ast}=m$.

We have obtained

\begin{theorem}[Conductor of constants 2]
\label{T2.5.AA} Let $ K$ be a finite abelian extension of
$k$. Let $n,m\in{\ma N}$ and $N\in R_T$ be such that
$K\subseteq {_n\cicl N{}_m}$ and such that $m$ is
minimum with this property. Let $t=f_{\infty}( K/k)
=f_{\infty}(K/F)$ be the degree
of the infinite primes of $K$. Let $R=
K_m\cap {_n\cicl N{}}$ and $d^{\ast}=\f{RK/K}$.
Then
\begin{gather*}
m=t e_{\infty}( K/F)=td^{\ast}=\f{RK/k}.
\end{gather*}
In particular
$$
d^{\ast}=\f{RK/K}=e_{\infty}( K/F).
\eqno{\fin}
$$
\end{theorem}

\begin{remark}\label{R2.6AA}{\rm{
From Theorems \ref{T2.1} and \ref{T2.4AA} follows that if
$K\subseteq {_n\cicl N{}}_m$, then $\g K \subseteq
{_n\cicl N{}}_m$. In particular the conductors of constants
of $K$ and of $\g K$ are the same.
}}
\end{remark}

\section{Genus fields of subfields of cyclotomic 
function fields}\label{S4}

For an abelian extension $K/k$, the description of $\g K$
depends on the description of $\g E$ (Theorem \ref{T2.1}).
In this section we present some details in order to find $\g E$.
For  the results and notation on Dirichlet characters we use,
we refer to \cite[Chapter 12]{Vil2006}.
Here $K$ denotes a field $k\subseteq K\subseteq k(\Lambda_N)$
for some $N\in R_T$ and $k={\ma F}_q(T)$.

\begin{remark}\label{R4.0}{\rm{
Let $k\subseteq K\subseteq k(\Lambda_N)$ and let $X$ be the
group of Dirichlet characters associated to $K$. If $L$ is the
field associated to $\prod_{P\in R_T^+}X_P$, then
\[
\g K=L^{\mc D},
\]
where ${\mc D}$ is the decomposition group of any prime ${\eu p}\in
\S K$ in $L/K$.
}}
\end{remark}

\begin{proposition}\label{P4.1}
With the notation as above, let $X$ be the group of Dirichlet
characters corresponding to $K$. Fix $P\in R_T^+$. Let $Y$ be
a group of Dirichlet characters such that $Y=Y_P$, that is, for
any $\chi\in Y$, the conductor of $\chi$ is a power of $P$:
${\mc F}_{\chi}=P^{\alpha_{\chi}}$ for some $\alpha_{\chi}\in {\ma N}
\cup\{0\}$. Let $L$ be the field associated to $\langle X,Y\rangle$,
that is, if $F$ is the field associated to $Y$, then $L=KF$. 
If $KF/K$ is unramified at $P$, then $Y\subseteq 
X_P$.
\end{proposition}

\begin{proof}
We have $|\langle X,Y\rangle_P|=e_P(KF/k)=e_P(KF/K)
e_P(K/k)=e_P(K/k)=|X_P|$. Since $X_P\subseteq \langle
X,Y\rangle_P$, it follows that $X_P=\langle X,Y\rangle_P$.
Since $Y_P\subseteq \langle X,Y\rangle_P$, the result follows.
\end{proof}

\begin{corollary}\label{C4.2}
If $|Y|=|X_P|$, then $Y=X_P$. $\fin$
\end{corollary}

We apply Proposition \ref{P4.1} to Kummer extensions of $k$
and to finite abelian $p$--extensions of $k$.

\subsection{Kummer extensions}\label{S4.1}

Let $K=k(\sqrt[t]{\gamma D})$ be a Kummer extension with
$K\subseteq k(\Lambda_D)$, that is, $t|q-1$, $D\in R_T$ is
a monic polynomial, $D$ is $t$--power free and
$\gamma=(-1)^{\deg D}$.
Say $D= P_1^{\alpha_1}\cdots P_r^{\alpha_r}$, $r\geq 1$,
$1\leq \alpha_i\leq t-1$, $1\leq i\leq r$, as a product of
powers of monic irreducible polynomials.
 Set $d_i:=\gcd(\alpha_i,t)$.
Then $\gcd\big(\frac{\alpha_i}{d_i},\frac{t}{d_i}\big)=1$.
Let ${\eu p}_i$ be a prime in $K$ above $P_i\in R_T^+$. Set
$\beta:=\sqrt[t]{\gamma D}$ so that $\beta^t=\gamma D=
\gamma P_1^{\alpha_1}\cdots P_r^{\alpha_r}$. We
have that $e_i:=e_{P_i}(K/k) = t/d_i$ (see 
\cite[Subsection 5.2]{MaRzVi2016}).

Let $F_i=k\Big(\sqrt[t/d_i]{(-1)^{\deg P_i^{\alpha_i/d_i}}
P_i^{\alpha_i/d_i}}\Big)$.
Set $\gamma_i=(-1)^{\deg P_i^{\alpha_i/d_i}}$. Let $X$ be the
group of Dirichlet characters associated to $K$. In fact $X$ is
a cyclic group of order $t$ and let $X=\langle \chi\rangle$. Let $Y$
be the group of Dirichlet characters associated to $F_i$. 
Then $Y=Y_{P_i}$ and
$|Y_{P_i}|=e_{P_i}(F_i/k)=t/d_i$ since $\gcd (t/d_i,\alpha_i/d_i)=1$,
and $|X_{P_i}|=e_{P_i}(K/k)=t/d_i=|Y_{P_i}|$. 

We will see that $KF_i/K$ is unramified at $P_i$. We have
\begin{align*}
KF_i&=k\big(\sqrt[t]{\gamma D}, \sqrt[t/d_i]{\gamma_i
P_i^{\alpha_i/d_i}}\big)=k\big(\sqrt[t]{\gamma D},\sqrt[t]{
\gamma_i^{d_i}P_i^{\alpha_i}}\big)\\
&= K\big(\sqrt[t]{(-1)^{\deg P_i^{\alpha_i}}P_i^{\alpha_i}}\big)=
K\Big(\sqrt[t]{\frac{\gamma D}{\gamma_i^{d_i} P_i^{\alpha_i}}}\Big),
\end{align*}
and $P_i\nmid \frac {D}{P_i^{\alpha_i}}$. Hence $P_i$ is unramified
in $KF_i/K$. Therefore $Y_{P_i}=X_{P_i}=Y$.

It follows that the field associated to the group $\prod_PX_P$ is $k(\xi_1,\ldots
\xi_r)$ where $\xi_i=\sqrt[t/d_i]{\gamma_i P_i^{\alpha_i/d_i}}$.

We have proved

\begin{theorem}\label{T4.3(1)}
Let $X$ be the group of Dirichlet characters associated to $K=k\big(\sqrt[
t]{\gamma D}\big)$ with $t\mid q-1$, $D\in R_T$ and is $t$--power free,
$D=P_1^{\alpha_1}\cdots P_r^{\alpha_r}$, $r\geq 1$, $1\leq \alpha_i
\leq t-1$, $1\leq i\leq r$, $\gamma =(-1)^{\deg D}$. Let $d_i=\gcd(t,
\alpha_i)$, $1\leq i\leq r$. Then the field associated to $\prod_PX_P=
\prod_{i=1}^rX_{P_i}$ is $L=k(\xi_1,\ldots,\xi_r)$ where $\xi_i
=\sqrt[t/d_i]{\gamma_i P_i^{\alpha_i/d_i}}$ and $\gamma_i=(-1)^{
\deg P_i^{\alpha_i/d_i}}$. That is, 
\[
L=k\Big(\sqrt[t]{(-1)^{\deg P_1^{\alpha_1}}P_1^{\alpha_1}} ,\ldots,
\sqrt[t]{(-1)^{\deg P_r^{\alpha_r}}P_r^{\alpha_r}}\Big)
\]
and the genus field of $K$ is $\g K=L^{\mc D}$, where ${\mc D}$
is the decomposition group of any prime ${\eu p}\in \S K$ in $L/K$.
$\fin$
\end{theorem}

\subsection{Abelian $p$--extensions}\label{S4.2}

We consider now $K=k(\vec y)$ where 
\[
\vec y^{p^u}\Witt -\vec y=
\vec \delta_1\Witt +\cdots \Witt + \vec \delta_r
\]
 with 
$\vec \delta_i=(\delta_{i,1},\ldots,\delta_{i,v})$ for some $v\in{\ma N}$,
$\delta_{i,j}=
\frac{Q_{i,j}}{P_i^{e_{i,j}}}$, $e_{i,j}\geq 0$, $Q_{i,j}\in R_T$.
Here we assume that ${\ma F}_{p^u}\subseteq k_0={\ma F}_q$
and that $K\subseteq k(\Lambda_N)$ for some $N\in R_T$.

Let $X$ be the group of characters associated to $K$. According
to Schmid \cite{Sch36}, the ramification index of $P_i$ in $K/k$ is
determined by the first index $j$ such that we may write $\delta_{i,j}=
\frac{Q_{i,j}}{P_i^{e_{i,j}}}$ with $\gcd(Q_{i,j},P_i)=1$, $e_{i,j}>0$
and $\gcd(e_{i,j},p)=1$.

In other words, the ramification index of $P_i$ at $K/k$ depends only
on $\vec \delta_i$ and not on $\vec \delta_1,\ldots, \vec\delta_{i-1},
\vec \delta_{i+1},\ldots, \vec \delta_r$. Therefore, if $Y$ is the group
of characters associated to 
\[
F_i=k(\vec y_i)\quad \text{with}\quad \vec y_i^{p^u}
\Witt -\vec y_i=\vec\delta_i,\quad 1\leq i\leq r,
\]
we have $|X_{P_i}|=|Y|=|Y_{P_i}|$.
Furthermore, the extension $KF_i=k(\vec y,\vec y_i)=k(\vec y, 
\vec y\Witt - \vec y_i)=K(\vec y\Witt - \vec y_i)$ is unramifed at
$P_i$ over $K$. It follows that the field associated to 
$\prod_P X_P=\prod_{i=1}^r X_{P_i}$
is $k(\vec y_1,\ldots, \vec y_r)$. Here the decomposition
group ${\mc D}$ is trivial.

Then, we have

\begin{theorem}\label{T4.3}
With the conditions as above, if $K=k(\vec y)$, then
the field associated to $\prod_P X_P=\prod_{i=1}^r X_{P_i}$
is 
\[
L=k(\vec y_1,\ldots, \vec y_r)
\]
and the genus field of $K$ is also
$$
\g K=k(\vec y_1,\ldots, \vec y_r). \eqno{\fin}
$$
\end{theorem}

\section{Explicit description of genus fields 
of abelian $p$--extensions}\label{S3}

Let $K/k$ be a finite abelian $p$--extension. Recall that $k=k_0(T)$
with $k_0={\ma F}_q$, say $q=p^l$.
We will assume that ${\ma F}_{p^u}\subseteq k_0$, that is, $u\mid l$. 

Then we  have
\[
\Gal(K/k)\cong \big({\ma Z}/p^{\alpha_1}{\ma Z}\big)\times
\cdots\times \big({\ma Z}/p^{\alpha_u}{\ma Z}\big) \quad
\text{with}\quad 1\leq \alpha_1\leq \cdots\leq \alpha_u=v.
\]
There exist $\vec w_1, \ldots, \vec w_u\in W_v(\bar{k})$ such that
$\vec w_i^p\Witt - \vec w_i=\vec\xi_i\in W_v(k)$, with
$K=k(\vec w_1,\cdots,\vec w_v)$. We also have
that there exists $\vec y_0\in W_v(\bar{k})$ such that 
$K=k(\vec y_0)$ with
\[
\vec y_0^{p^u}\Witt - \vec y_0=\vec\xi_0 \quad\text{for some}
\quad \vec \xi_0\in W_v(k)
\]
(see 
\cite[Theorem 8.5]{BaJaRzVi2016}). Here $\bar{k}$ denotes
an algebraic closure of $k$.

Let $P_1,\ldots,P_r\in R_T^+$ be the finite primes in $k$
ramified in $K$.
From \cite[Theorem 8.10]{BaJaRzVi2016} it follows that we may 
decompose $\vec\xi_0$ as
\begin{gather}\label{Eq3.0}
\vec \xi_0={\vec\delta}_1
\Witt + \cdots \Witt + {\vec\delta}_r \Witt + \vec\gamma,
\end{gather}
where $\delta_{i,j}=\frac{Q_{i,j}}{P_i^{e_{i,j}}}$, $e_{i,j}\geq 0$, $Q_{i,j}\in R_T$
and if $e_{i,j}>0$, then $e_{i,j}=
\lambda_{i,j}p^{m_{i,j}}$, $\gcd (\lambda_{i,j},p)=1$,
$0\leq m_{i,j}< n$, $\gcd(Q_{i,j},P_i)=1$ and
$\deg (Q_{i,j})<\deg (P_i^{e_{i,j}})$, and $\gamma_j=f_j(T)\in R_T$ with
$\deg f_j=\nu_j p^{m_j}$ and $\gcd(q,\nu_j)=1$, $0\leq m_j<n$
when $f_j\not\in k_0$.

If the ramification index of $P_i$ is $p^{a_i}<p^v$, we may
write $\vec \delta_i=(\delta_{i,1},\ldots,\delta_{i,v})=(0,\ldots,0,\delta_{
i,(v-a_i+1)},\ldots,\delta_{i,v})$. In particular $\p$ decomposes fully
in $k(\vec y_i)/k$, where
$\vec y_i^{p^u}\Witt - \vec y_i=\vec \delta_i$
(see \cite[Theorem 8.13]{BaJaRzVi2016}).

Let $\vec z^{p^u}\Witt - \vec z=\vec \gamma$. In $k(\vec z)/k$ the only
possible ramified prime is $\p$. Note that if 
\[
\vec y=\vec y_1\Witt +\cdots\Witt + \vec y_r, \quad \text{then}\quad 
\vec y^{p^u}\Witt -\vec y=\vec\xi_0\Witt - \vec \gamma=\vec\delta_1\Witt +
\cdots\Witt + \vec\delta_r
\]
and $\p$ decomposes fully in $k(\vec y)/k$.

The first main result of this section is

\begin{theorem}\label{T3.1}
With the above notation, let $E=KM\cap k(\Lambda_N)$. 
Then $E=k(\vec y)$, $\g E=k(\vec y_1,\ldots,
\vec y_r)$ and 
\[
\g K=k(\vec y_1,\ldots,\vec y_r, \vec z).
\]
\end{theorem}

\begin{proof}
From the Galois correspondence
$EM=KM$. To prove that $E=k(\vec y)$ is equivalent to show that
$k(\vec y)M=KM$ since $k(\vec y)\subseteq k(\Lambda_N)$.

Now, $k(\vec z)\subseteq M$ since $M=L_n{\ma F}_{q^m}(T)$
codifies all the inertia and all the ramification, which is totally wild, of $\p$.
We have
\begin{gather*}
k(\vec y)M=k(\vec y)k(\vec z)M\supseteq k(\vec y\Witt +\vec z)M=KM.
\intertext{Also,}
KM=Kk(\vec z)M=k(\vec y_0)k(\vec z)M\supseteq k(\vec y_0\Witt - \vec z)M=
k(\vec y)M.
\intertext{Thus}
KM=k(\vec y)M\quad \text{and} \quad E=k(\vec y).
\end{gather*}

From \cite{MaRzVi2013} (see also Theorem \ref{T4.3}) 
we obtain $\g E=k(\vec y_1,\ldots,\vec y_r)$.
Finally
\begin{align*}
\g K&=\g E K=k(\vec y_1,\ldots, \vec y_r)k(\vec y_0)=k(\vec y_1,\ldots,\vec y_r)
k(\vec y_0\Witt -\vec y_1\Witt -\cdots\Witt -\vec y_r)\\
&=k(\vec y_1,
\ldots,\vec y_r)k(\vec z)=k(\vec y_1,\ldots,\vec y_r,\vec z).
\end{align*}
This finishes the proof.
\end{proof}

\begin{remarks}\label{R3.2}{\rm{
\l
\item Observe that with the above conditions $[k(\vec y_i):k]=e_{P_i}(K/k)$
and $[k(\vec z):k]=\e{K/k}\cdot \f{K/k}$.

\item Note that the proof of Theorem \ref{T3.1} works even in the case that
$\vec \delta_i$ and $\vec \gamma$ are not in the reduced form
described above. We only need that in each extension $\vec y_i^{p^u}
\Witt -\vec y_i=\vec \delta_i$, $1\leq i\leq r$ and $\vec z^{p^u}\Witt -\vec z=
\vec \gamma$ there is at most one prime ramifying.
\end{list}
}}
\end{remarks}

From Theorem \ref{T2.2}, the cases of
Artin--Schreier and Witt extensions, and elementary 
abelian $p$--extensions are an immediate consequence of
Theorem \ref{T3.1}.

\begin{corollary}[Theorems 5.4 and 
5.7 of \cite{MaRzVi2013}]\label{C3.3} Let $k=k_0(T)$.
\l
\item Let $K=k(y)$ with 
\[
y^p-y=\alpha=\sum_{i=1}^r\frac{Q_i}{P_i^{e_i}} + f(T),
\]
where $P_i\in R_T^+$, $Q_i\in R_T$, 
$\gcd(P_i,Q_i)=1$, $e_i>0$, $p\nmid e_i$, $\deg Q_i<
\deg P_i^{e_i}$, $1\leq i\leq r$, $f(T)\in R_T$,
with $p\nmid \deg f$ when $f(T)\not\in k_0$.

Then 
\[
\g K=k(y_1,\ldots,y_r,\beta),
\]
where
$y_i^p-y_i=\frac{Q_i}{P^{e_i}}$, $1\leq i\leq r$ and
$\beta^p-\beta=f(T)$.

\item Let $K=k(\vec y)$ where
\[
\vec y^p\Witt -\vec y=\vec \beta={\vec\delta}_1\Witt + \cdots \Witt + {\vec\delta}_r
\Witt + \vec\mu,
\]
with $\delta_{i,j}=\frac{Q_{i,j}}{P_i^{e_{i,j}}}$, $e_{i,j}\geq 0$, $Q_{i,j}\in R_T$,
$\gcd(Q_{i,j},P_i)=1$ and if $e_{i,j}>0$, then $p\nmid e_{i,j}$, and 
$\deg (Q_{i,j})<\deg (P_i^{e_{i,j}})$, and $\mu_j=f_j(T)\in R_T$ with
$p\nmid \deg f_j$ when $f_j\not\in k_0$.

 Then
\[
\g K=k({\vec y}_1,\ldots,{\vec y}_r,\vec z),
\]
where ${\vec y}_i^p\Witt -{\vec y}_i ={\vec \delta}_i$, $1\leq i\leq r$
and ${\vec z}^p\Witt -\vec z=\vec\mu$. 

\item Assume that ${\ma F}_{p^u}\subseteq k_0$. Let $K=k(y)$ with 
\[
y^{p^u}-y=\alpha=\sum_{i=1}^r\frac{Q_i}{P_i^{e_i}} + f(T),
\]
where $P_i\in R_T^+$, $Q_i\in R_T$ and $f(T)\in k_0[T]$.

Then 
\[
\g K=k(y_1,\ldots, y_r, z),
\]
where $y_i^{p^u}-y_i=\frac{Q_i}{P_i^{e_i}}$,
$1\leq i\leq r$ and $z^{p^u}-z=f(T)$. \fin

\end{list}
\end{corollary}

\section{General finite abelian extensions of $k$}\label{S5}

Up to now we have given the explicit description of the genus fields
of abelian $p$--extensions $K$ of $k=k_0(T)$ where $k_0={\ma F}_q$ is
such that ${\ma F}_{p^u}\subseteq k_0$ and $K=k(\vec y)$ and $\vec y$
is given by an equation of the form $\vec y^{p^u} \Witt - \vec y=\vec \beta
\in W_m(k)$. When ${\ma F}_{p^u}\nsubseteq k_0$ the field $K$ 
cannot be given by this type of equations.

In this section we give explicitly the description of $\g K$ where
$K/k$ is a finite abelian extension of degree $t$ with $\gcd(t,q-1)=1$.
The case $t\mid q-1$ is treated in Subsection \ref{S4.1}.

\begin{remark}\label{R5.1}{\rm{
For any abelian extension $K/k$ of degree $t$ with $\gcd(t,q-1)=1$,
we have that if $E=KM\cap k(\Lambda_N)$, then $[E:k]\mid t$
(see (\ref{Eq4})). If $X$ is the set of Dirichlet characters of $E$,
we have $\gcd(|X|,q-1)=\gcd([E:k],q-1)=1$. Since for any $\chi\in X$
and any $P\in R_T^+$, we have that $\chi_P^{|X|}=1$, we obtain
that $\gcd([\g E:k],q-1)=1$. In particular $H=\{1\}$. Therefore
$\g K=\g E K$.
}}
\end{remark}

In general if $K_1$ and $K_2$ are two finite extensions of $k$
we have
\[
\g{(K_1)}\g{(K_2)}\subseteq \g{(K_1K_2)},
\]
but we may have $\g{(K_1)}\g{(K_2)}\subsetneq \g{(K_1K_2)}$. 
In fact, let $q>2$ and $P,Q,R,S\in R_T$ 
be four different monic polynomials in $R_T$. Set
$L_1:=k(\Lambda_{PQ})^+$ and $L_2:=k(\Lambda_{RS})^+$.
Then $\g{(L_i)}=L_i$, $i=1, 2$. Therefore $\g{(L_1)}\g{(L_2)}
=L_1L_2$. On the other hand, $\g{(L_1L_2)} = k(\Lambda_{PQRS})^+
\supsetneqq \g{(L_1)}\g{(L_2)}$, see \cite[Remark 3.7]{MaRzVi2016}.

We will show that for finite
abelian extensions of $k$ of degree relatively 
prime to $q-1$ we have equality.
In particular if $K_1$ and $K_2$ are finite abelian $p$--extensions
of $k$, we have equality.

For a subfield $K\subseteq k(\Lambda_N)$ for some $N\in R_T$,
denote by $\g {K'}$ the maximal abelian extension of $K$ contained
in $k(\Lambda_N)$, unramified at the finite primes. We have
(see Remark \ref{R4.0})
\begin{gather}\label{Eq3.1}
\g K=(\g {K'})^{\mc D},
\end{gather}
where ${\mc D}$ is the decomposition group of 
any element of $\S K$ in $\g {K'}/K$.

Consider $K_i\subseteq k(\Lambda_N)$, $i=1,2$ and let $X_i$
be the group of Dirichlet characters associated to $K_i$. Therefore
$Y=X_1X_2=\langle X_1,X_2\rangle$ is the group of Dirichlet
characters associated to $L=K_1K_2$. Let $P\in R_T^+$. It is easy
to see that
\begin{gather*}
\langle X_1,X_2\rangle_P=\langle (X_1)_P,(X_2)_P\rangle,
\intertext{so that we obtain}
\prod_{P\in R_T^+}Y_P=\prod_{P\in R_T^+}
\langle X_1,X_2\rangle_P=\Big(\prod_{P\in R_T^+}(X_1)_P\Big) \cdot
\Big(\prod_{P\in R_T^+} (X_2)_P\Big).
\intertext{It follows that}
\g {(K_1)}' \g {(K_2)}'=\g{(K_1K_2)}'.
\end{gather*}

We have proved

\begin{proposition}\label{P3.4} For $K_i\subseteq k(\Lambda_N)$, 
$i=1,2$, we have
$$
\g {(K_1)}' \g {(K_2)}'=\g{(K_1K_2)}'. \eqno{\fin}
$$
\end{proposition}

\begin{corollary}\label{C3.5} Let $K_i\subseteq k(\Lambda_N)$, 
$i=1,2$ be such that $K_1/k$ and $K_2/k$ are finite abelian extensions
of degrees relatively prime to $q-1$. Then
$\g {(K_1)} \g {(K_2)}=\g{(K_1K_2)}$.
\end{corollary}

\begin{proof}
Since the decomposition groups of $\p$ in $K_1/k$, in $K_2/k$ and
in $K_1K_2/k$ are the unit group, 
it follows from (\ref{Eq3.1}) that $\g {(K_i)}=
\g{(K_i)}'$, $i=1,2$ and $\g{(K_1K_2)}=\g{(K_1K_2)}'$. 
The result follows from Proposition \ref{P3.4}.
\end{proof}

\begin{corollary}\label{C3.6}
Let $K_i/k$, $i=1,2$ be two finite abelian extensions of
degrees relatively prime to $q-1$. Then
\[
\g {(K_1)} \g {(K_2)}=\g{(K_1K_2)}.
\]
\end{corollary}

\begin{proof}
Let $k_0={\ma F}_{p^l}$,
$K_i\subseteq L_nk(\Lambda_N) 
{\ma F}_{p^{lm}}(T)$, $i=1,2$, and let
$M:=L_n{\ma F}_{p^{lm}}(T)$. Set
$E_i:=K_iM\cap k(\Lambda_N)$, $i=1,2$ and $E:=K_1K_2M
\cap k(\Lambda_N)$. Using the Galois correspondence,
it can be proved that $E=E_1E_2$. 

From Corollary \ref{C3.5} we have 
$\g E=\g {(E_1)} \g {(E_2)}$. Therefore
\begin{align*}
\g {(K_1)} \g {(K_2)}&=\g{(E_1)}K_1\cdot \g{(E_2)}K_2=
\g {(E_1)}\g{(E_2)} \cdot K_1K_2\\
&=\g E \cdot K_1K_2=
\g{(K_1K_2)}.
\end{align*}
Thus $\g {(K_1)} \g {(K_2)}=\g{(K_1K_2)}$.
\end{proof}

\begin{corollary}\label{C3.7}
Let $K_i/k$, $i=1,2$ be two finite abelian $p$--extensions. Then
$$
\g {(K_1)} \g {(K_2)}=\g{(K_1K_2)}. \eqno{\fin}
$$
\end{corollary}

As a consequence we obtain the description of the genus field
of a finite abelian $p$--extension of $k$.

\begin{corollary}\label{C3.8} Let $K/k$ be a finite 
abelian $p$--extension with Galois group $\Gal(K/k)=G
\cong G_1\times\cdots\times G_s$ with $G_i\cong
{\ma Z}/p^{\alpha_i}{\ma Z}$, $1\leq i\leq s$. Let $K$
be the composite $K=K_1\cdots K_s$ such that
$\Gal(K_i/k)\cong G_i$. Let $P_1,\ldots, P_r$
be the finite primes ramified in $K/k$.
Let $K_i=k(\vec w_i)$ be given
by the equation
\begin{gather*}
\vec w_i^p\Witt - \vec w_i=\vec \xi_i,\quad 1\leq i\leq s.
\intertext{Write each $\vec\xi_i$ as in {\rm{(\ref{Eq3.0})}} that is,}
\vec \xi_i={\vec\delta}_{i,1}
\Witt + \cdots \Witt + {\vec\delta}_{i,r} \Witt + \vec\gamma_i,
\intertext{such that all the components of $\vec\delta_{i,j}$ are written
so that the degree of the numerator is less than the degree
of the denominator, the support of
the denominator is at most $\{P_j\}$ and the components
of $\vec\gamma_i$ are polynomials. Let}
\vec w_{i,j}^p\Witt -\vec w_{i,j}=\vec\delta_{i,j},\quad
1\leq i\leq s,\quad 1\leq j\leq r
\intertext{and}
\vec z_i^p\Witt -\vec z_i=\vec \gamma_i, \quad 1\leq i\leq s.
\intertext{Then}
\g K=k\big(\vec w_{i,j},\vec z_i\mid 1\leq i\leq s, 1\leq j\leq r\big).
\end{gather*}
\end{corollary}

\begin{proof}
It is a consequence of Remarks \ref{R3.2} (b), Corollary \ref{C3.3} (b) and
Corollary \ref{C3.7}.
\end{proof}

Next, we consider a cyclic extension $K/k$ of degree $t$ such that
$\gcd(t,p(q-1))=1$. We have that $E=KM\cap k(\Lambda_N)$ satisfies that
$[E:k]$ is relatively prime to $q-1$. Hence $\g {E'}=\g E$ and $\g K=
\g E K$. Thus, we have to describe $\g E$.

\begin{proposition}\label{P3.9} Let $E\subseteq k(\Lambda_N)$ be
a cyclic extension of $k$ of degree $t$ relatively prime to $p(q-1)$.
Let $P_1,\ldots,P_r\in R_T^+$ be the primes in $k$ ramifying in $E$.
Then 
\[
\g E=\prod_{j=1}^r F_j,
\]
where $k\subseteq F_j\subseteq k(\Lambda_{P_j})$ is the subfield
of degree $a_j$ over $k$, $a_j$ is the order of $\chi_{P_j}$,
and $\chi$ is the character associated to $E$.
\end{proposition}

\begin{proof}
It follows from the fact that $X=\langle \chi\rangle$ is the group
of Dirichlet characters associated to $E$, $\g E$ is the field
corresponding to $\prod_{j=1}^r X_{P_j}$, $X_{P_j}=
\langle\chi_{P_j}\rangle$ and $F_j$ is the field associated to
$\chi_{P_j}$.
\end{proof}

We have our final main result.

\begin{theorem}\label{T3.10} Let $K/k$ be an abelian extension of degree
$t$ with $\gcd (t,q-1)=1$. Let $P_1,\ldots, P_r\in R_T^+$ be the
primes in $k$ ramifying in $K$.
Let $E=KM\cap k(\Lambda_N)=E_0E_1\cdots E_s$
where $E_i/k$ is a cyclic extension of degree $t_i$, $\gcd(t_i,p(q-1))=1$,
$1\leq i\leq s$ and $E_0/k$ is an abelian $p$--extension. 
Then 
\[
\g K=\g E K,\quad\text{where}\quad
\g E=\g {(E_0)}\g {(E_1)}\cdots \g {(E_s)},
\]
$\g {(E_0)}$ is given by Corollary {\rm{\ref{C3.8}}} and $\g {(E_i)}=
\prod_{j=1}^r F_{i,j}$ is given by
Proposition {\rm{\ref{P3.9}}}, $1\leq i\leq s$. 

Furthermore, let $b_{i,j}:=[F_{i,j}:k]$. Then $F_j:=\prod_{i=1}^s F_{i,j}$
is the subfield of $k(\Lambda_{P_j})$ of degree $b_j:=\lcm[b_{i,j},
1\leq i\leq s]$ over $k$. We have
$$
\g K=\g{(E_0)}\Big(\prod_{j=1}^r F_j\Big) K. \eqno{\fin}
$$
\end{theorem}

\end{document}